\documentclass[12pt]{amsart}
\usepackage{times}
\DeclareMathAlphabet{\mathpzc}{OT1}{pzc}{m}{it}

\usepackage[T1]{fontenc}
\usepackage{dsfont}
\usepackage{mathrsfs}
\usepackage[colorlinks]{hyperref}
\usepackage{xcolor}
\usepackage[a4paper,asymmetric]{geometry}
\usepackage{mathscinet}
\usepackage{fullpage}
\usepackage{latexsym}
\usepackage{graphicx}
\usepackage{epstopdf}
\usepackage{amsthm}
\usepackage{amssymb}
\usepackage{amsfonts}
\usepackage{amsmath}
\usepackage{subfigure}
\newtheorem{theorem}{Theorem}[section]
\newtheorem{lemma}[theorem]{Lemma}

\theoremstyle{definition}
\newtheorem{definition}[theorem]{Definition}
\newtheorem{example}[theorem]{Example}
\newtheorem{remark}[theorem]{Remark}

\newtheorem*{assA*}{Assumption A}

\numberwithin{equation}{section}

\newcommand{\dif}{\mathrm{d}}

\begin{document}

\title[Stable]
{Stable Approximation for Call Function Via Stein's method}

\author[P.~Chen]{Peng Chen}
\address[P.~Chen]{School of Mathematics, Nanjing University of Aeronautics and
Astronautics, Nanjing 211106, China}
\email{chenpengmath@nuaa.edu.cn}

\author[T.~Qi]{Tianyi Qi}
\address[T.~Qi]{School of Mathematics, Nanjing University of Aeronautics and
Astronautics, Nanjing 211106, China}
\email{qitianyi@nuaa.edu.cn}

\author[T.~Zhang]{Ting Zhang \textsuperscript{*}}
\thanks{Ting Zhang\textsuperscript{*}: Corresponding author. Email: zhangtstat@nuist.edu.cn }
\address[T.~Zhang\textsuperscript{*}]{School of Mathematics and Statistics, Nanjing University of Information Science and Technology, Nanjing 210044, China}
\email{zhangtstat@nuist.edu.cn}

\makeatletter
\@namedef{subjclassname@2020}{%
  \textup{2020} Mathematics Subject Classification}
\makeatother

\subjclass[2020]{}

\begin{abstract}
Let $S_{n}$ be a sum of independent identically distribution random variables with finite first moment and $h_{M}$ be a call function defined by $g_{M}(x)=\max\{x-M,0\}$ for $x\in\mathbb{R}$, $M>0$. In this paper, we assume the random variables are in the domain $\mathcal{R}_{\alpha}$ of normal attraction of a stable law of exponent $\alpha$, then for $\alpha\in(1,2)$, we use the Stein's method developed in \cite{CNX21} to give uniform and non uniform bounds on $\alpha$-stable approximation for the call function without additional moment assumptions. These results will make the approximation theory of call function applicable to the lower moment conditions, and greatly expand the scope of application of call function in many fields.
\end{abstract}

\maketitle


\section{Introduction}

For a fixed constant $M>0$, the call function $g_{M}$ is defined by
\begin{align*}
g_{M}(x)=(x-M)_{+},\qquad x\in\mathbb{R},
\end{align*}
where $(x-M)_{+}=\max\{x-M,0\}$. In recent years, the call function has been used in many aspects, such as, risk theory \cite{Jag84,WCWK12}, finance \cite{CZHW24,VX24}, etc. For example, one of the representative applications of call function is the collateralized debt obligation (CDO), which belongs to the financial instrument and plays a crucial role in the financial crisis in the USA. For more details on the CDO, we refer the reader to \cite{ACHL12,PJYT24,NY09,NYD08,AN21,KN20,NK24,NKN22} and the references therein.

For a long time, many authors have been working on the approximation for the call function and one of the commonly used methods in recent years is the Stein's method. The Stein's method was proposed by Stein in \cite{Ste72} when dealing with the problem of normal approximation. Now, the Stein's method has become a commonly used tool for obtaining error bounds between two probability distributions. For Stein's method, we refer the reader to \cite{BRZ24,Bon20,Cha07,CGS10,FLS24,NP09,Rol22,Son20}. To be specific, \cite{NYD08} used Stein's method and zero bias transformation to obtain a bound on Gaussian and Poisson approximation, then a more thorough analysis of the approximation was conducted using CDO. Subsequently, the approximations are further investigated in \cite{PJYT24,NY09,KN20,NK24,NKN22} and better non uniform upper bounds are obtained, as well as some better upper bounds with correction terms. Moreover, \cite{AN21} studied the negative binomial approximation for call function by Stein's method.

However, it is a pity that all kinds of approximation for call function studies above have the assumption that the second moments or variances are bounded, and some non uniform results even require the existence of higher-order moments. These conditions rule out many heavy-tailed distributions in theory and applications. In recent years, the heavy-tailed distribution, especially for the stable distribution, have been widely used in various fields, such as statistics, finance and economics \cite{AFV10,M63,N14,PPU21,RM00}. Hence, the research on the stable approximation of call function will further expand the use of heavy-tailed distributions in these fields.

As an important part of limit theorems, stable distribution and stable central limit theorem (CLT) play a significant role in probability theory, which have been used in many fields, such as economics \cite{M63,N14}, deep learning \cite{FFP23,JLL23} and so on. Recently, new progress has been made in the study of stable approximation by Stein's method. \cite{lihu} considered the symmetric $\alpha$-stable distribution with $\alpha\in(1,2)$, and first used the Stein's method and the Stein kernel to obtain the optimal convergence rate of the stable CLT in the Wasserstein-1 distance. Along this research direction, \cite{CNX21} generalized the method to the asymmetric case with the  help of the zero-biased coupling and Taylor-like extension, then the method was further extended to the non-integrable case \cite{CNXYZ22} and multivariate case \cite{CNXYZ22}. For more research on stable approximation, we refer the reader to \cite{AH19,AH192,banis,BU20,DN02,HDC24} and the references therein.

In this paper, we will use the Stein's method developed in \cite{CNX21,lihu} to give uniform and non uniform bounds on $\alpha$-stable approximation for the call function without additional moment assumptions. In particular, for the symmetric case (that is, $\delta=0$), we further improve the non uniform upper bounds with the help of heat kernel estimates. These results will make the approximation theory of call function applicable to the lower moment conditions, and greatly expand the scope of application of call function in CDO and other fields.

Now, we first give the definition of the $\alpha$-stable distribution.

\begin{definition}\label{defstable}
Let $\alpha\in(1,2)$, $\sigma\geq 0$ and $\delta\in [-1,1]$ be real numbers. We say that $Y$ is distributed according to the \emph{$\alpha$-stable law with parameters $\sigma$ and $\delta$}, and we write $Y\sim S_\alpha(\sigma,\delta)$, to indicate that, for all $\lambda\in(-\infty,\infty)$,
\[
\mathbb{E}[e^{i\lambda Y}]=
\begin{array}{lll}
{\rm exp}\big\{-\sigma^\alpha|\lambda|^\alpha(1-i\,\delta\,{\rm sign}(\lambda)\,\tan\frac{\pi\alpha}2)\big\}
\end{array}
\]
\end{definition}

It is easy to check that $Y/\sigma\sim  S_\alpha(1,\delta)$  if $Y\sim  S_\alpha(\sigma,\delta)$. Hence, we will only consider the stable distributions for $\sigma=1$.

\subsection{Main assumptions and theorems}
In this paper, we assume the random variables are in the domain $\mathcal{R}_{\alpha}$ of normal attraction of a stable law of exponent $\alpha$, which is defined as follows:

\begin{definition}
If $X$ has a distribution function of the form
\begin{equation}\label{px}
F_{X}(x)=\big(1-\frac{A+B(x)}{|x|^{\alpha}}(1+\delta)\big){\bf 1}_{[0,\infty)}(x)+\frac{A+B(x)}{|x|^{\alpha}}(1-\delta){\bf 1}_{(-\infty,0)}(x)\big),
\end{equation}
where $\alpha\in(1,2)$, $A>0$, $\delta\in[-1,1]$ and  $B: \mathbb{R}\to\mathbb{R}$ is a bounded function vanishing at $\pm\infty$, then we say that $X$ is in the domain $\mathcal{R}_\alpha$ of normal attraction of a stable law of exponent $\alpha$.
\end{definition}

In (\ref{px}), the function $B$ is supposed bounded, that is, there exists $L>0$ such that $|B(x)|\leq L$. To be specific,
we assume there exist constants $L>0$ and $\gamma\geq 0$ such that
\begin{align}\label{bound}
|B(x)|\leq\frac{L}{|x|^{\gamma}},\quad x\neq0.
\end{align}

Notice that making $\gamma=0$ in (\ref{bound}) simply means that we do not want to make any extra assumption on $B$ defined in (\ref{px}).

To facilitate the formulation of the theorem, we first define two constants, which originate from the heat kernel estimates of the $\alpha$-stable process. Define
\begin{align}\label{eta1}
\eta_{1,\alpha,\delta}=\max\left\{\frac{\Gamma(\frac{1}{\alpha})}{\pi\alpha},\frac{(\alpha-1)\left(1+\delta\,{\rm tan}\frac{\pi\alpha}{2}\right)\left(2+\delta\,{\rm tan}\frac{\pi\alpha}{2}\right)\Gamma\left(\frac{\alpha-1}{\alpha}\right)}{\pi}\right\},
\end{align}
\begin{align}\label{eta2}
\eta_{2,\alpha,\delta}={\rm Beta}\left(\frac{2}{\alpha},1-\frac{1}{\alpha}\right)\max\left\{\frac{\Gamma(\frac{2}{\alpha})}{\pi\alpha},\frac{\left(1+\delta\,{\rm tan}\frac{\pi\alpha}{2}\right)\left(1+2\alpha+\alpha\delta\,{\rm tan}\frac{\pi\alpha}{2}\right)}{\pi}\right\},
\end{align}
where the Beta function ${\rm Beta}(u,v)=\int_{0}^{1}r^{u-1}(1-r)^{v-1}d r$ for any $u,v>0$.

Now, we state the first theorem.
\begin{theorem}[Uniform bound]\label{thm1}
Let $X_1,X_2,\ldots$ be independent and identically distributed random variables defined on a common probability space, and suppose that $X_1$ has
a distribution of the form (\ref{px}) with $B(x)$ satisfying (\ref{bound}).
Set $\sigma=\left(A\alpha\int_{-\infty}^{\infty} \frac{1-\cos y}{|y|^{1+\alpha}}dy\right)^\frac1\alpha$ and
\begin{equation}\label{sn}
S_n=\frac{1}{\sigma n^{\frac{1}{\alpha}}}\sum_{i=1}^{n}(X_i-\mathbb{E}[X_i]).
\end{equation}
Then we have
\begin{align*}
\left|\mathbb{E}\left(S_n-M\right)_{+}-\mathbb{E}\left(S_\alpha (1,\delta)-M\right)_{+}\right|\leq c_{1}{R_{n}},
\end{align*}
where
\begin{align}\label{Rn}
R_{n}=\begin{cases}
n^{1-\frac{2}{\alpha}},\quad &\gamma\in(2-\alpha,\infty),\\
n^{1-\frac{2}{\alpha}}\left|\log \left(\sigma n^{\frac{1}{\alpha}}\right)\right|,\quad &\gamma=2-\alpha,\\
n^{-\frac{(\alpha-1)\gamma}{\alpha(1-\gamma)}},\quad &\gamma\in(0,2-\alpha),\\
n^{1-\frac{2}{\alpha}}+n^{1-\frac{2}{\alpha}}\int_{-\sigma n^{\frac{1}{\alpha}}}^{\sigma n^{\frac{1}{\alpha}}}\frac{|B(x)|}{|x|^{\alpha-1}}dx+(\sup_{|x|\geq \sigma n^{\frac{1}{\alpha}}}|B(x)|)^{\alpha-1}, &\gamma=0,
\end{cases}
\end{align}
and
\begin{align*}
c_{1}=&
\left[\frac{16d_{\alpha}\mathbb{E}\left[\left|\xi_{1}-\mathbb{E}[\xi_{1}]\right|^{2-\alpha}\right]}{(2-\alpha)(\alpha-1)\sigma^{2-\alpha}}
+\frac{12\mathbb{E}\left[|\xi_{1}|\right]\left|\mathbb{E}[\xi_{1}]\right|}{\sigma^{2}}\right]\eta_{2,\alpha,\delta}\\
&+
\begin{cases}
\frac{8(2A)^{\frac{2}{\alpha}}}{\sigma^{2}}\left[\frac{2}{2-\alpha}+\frac{2L}{\alpha+\gamma-2}(2A)^{-\frac{\alpha+\gamma}{\alpha}}\right]\eta_{2,\alpha,\delta},
&\gamma\in(2-\alpha,\infty),\\
\frac{1}{\sigma^{2}}\left[4\left(\frac{4(2A)^{\frac{2}{\alpha}}}{2-\alpha}+\frac{8L}{\alpha-1}\right)\eta_{2,\alpha,\delta}
+\frac{8\alpha^{2}(A+L)-4L}{\alpha-1}\right],&\gamma=2-\alpha,\\
\sigma^{\frac{\alpha-\gamma}{\gamma-1}}\left[4\left(\frac{4(2A)^{\frac{2}{\alpha}}}{2-\alpha}+\frac{8L}{2-\alpha-\gamma}\right)\eta_{2,\alpha,\delta}
+\frac{8\alpha^{2}(A+L)-4L}{\alpha-1}\right], &\gamma\in(0,2-\alpha),\\
4\max\left\{\frac{2\alpha(2A)^{\frac{2}{\alpha}}}{(2-\alpha)\sigma^{2}},\frac{4}{\sigma^{2}},
\frac{8}{(2-\alpha)\sigma^{\alpha}}+\frac{2(2A)^{\frac{2}{\alpha}}}{\sigma^{\alpha}}+\frac{8\alpha^{2}(A+L)-4L}{4(\alpha-1)\eta_{2,\alpha,\delta}\sigma^{\alpha}}\right\}
\eta_{2,\alpha,\delta}, &\gamma=0.
\end{cases}
\end{align*}
\end{theorem}

Before giving the second main theorem, we define some constants as follows:
\begin{align}
&\eta_{3,\alpha,\delta}=\frac{\left(4\alpha+2\alpha^{\frac{2\alpha-2}{\alpha}}-1\right)\eta_{1,\alpha,\delta}}{(\alpha-1)},\nonumber\\
&\eta_{4,\alpha}=\max\left\{\frac{\Gamma(\frac{1}{\alpha})}{\alpha},\frac{\alpha 2^{\alpha-1}\sin\frac{\alpha\pi}{2}\Gamma(\frac{1+\alpha}{2})\Gamma(\frac{\alpha}{2})}{\pi^{\frac{3}{2}}}\right\}\frac{\left(4\alpha+2\alpha^{\frac{2\alpha-2}{\alpha}}-1\right)}{\alpha-1}.\label{6}
\end{align}

Now, we are in the position to give the non uniform bound:

\begin{theorem}[Non Uniform bound]\label{thm2}
Keep the same assumptions and notation as in Theorem \ref{thm1}. Then we have
\begin{align*}
\left|\mathbb{E}\left(S_n-M\right)_{+}-\mathbb{E}\left(S_\alpha (1,\delta)-M\right)_{+}\right|\leq c_{2,M}R_{n},
\end{align*}
where $R_{n}$ is defined by \ref{Rn} and
\begin{align*}
c_{2,M}=&
\frac{4d_{\alpha}\mathbb{E}\left[\left|\xi_{1}-\mathbb{E}[\xi_{1}]\right|^{2-\alpha}\right]}{(2-\alpha)(\alpha-1)\sigma^{2-\alpha}}\frac{\eta_{3,\alpha,\delta}}{M^{\frac{2(\alpha-1)}{3\alpha-1}}}
+\frac{3\mathbb{E}\left[|\xi_{1}|\right]\left|\mathbb{E}[\xi_{1}]\right|}{\sigma^{2}}\frac{\eta_{3,\alpha,\delta}}{M^{\frac{2(\alpha-1)}{3\alpha-1}}}\\
&+\eta_{3,\alpha,\delta}
\begin{cases}
\frac{2(2A)^{\frac{2}{\alpha}}}{\sigma^{2}}\left[\frac{2}{2-\alpha}+\frac{2L}{\alpha+\gamma-2}(2A)^{-\frac{\alpha+\gamma}{\alpha}}\right]M^{-\frac{2(\alpha-1)}{3\alpha-1}},&\gamma\in(2-\alpha,\infty),\\
\frac{1}{\sigma^{2}}\left[\left(\frac{4(2A)^{\frac{2}{\alpha}}}{2-\alpha}+\frac{8L}{\alpha-1}\right)
+q_{1}\right]\frac{2(\alpha-1)\log M}{3\alpha-1}M^{-\frac{2(\alpha-1)}{3\alpha-1}},& \gamma=2-\alpha,\\
\sigma^{\frac{\alpha-\gamma}{\gamma-1}}\left[\left(\frac{4(2A)^{\frac{2}{\alpha}}}{2-\alpha}+\frac{8L}{2-\alpha-\gamma}\right)
+q_{1}\right]M^{-\frac{2(\alpha-1)^{2}}{(3\alpha-1)(1-\gamma)}}, &\gamma\in(0,2-\alpha),\\
\max\left\{\frac{2\alpha(2A)^{\frac{2}{\alpha}}}{(2-\alpha)\sigma^{2}},\frac{4}{\sigma^{2}},
\frac{8}{(2-\alpha)\sigma^{\alpha}}+\frac{2(2A)^{\frac{2}{\alpha}}}{\sigma^{\alpha}}+q_{1}\right\}M^{-\frac{2(\alpha-1)^{2}}{(3\alpha-1)}},
&\gamma=0,
\end{cases}
\end{align*}
with
\begin{align}\label{q1}
q_{1}=\frac{8\alpha^{2}(A+L)-4\alpha L}{(\alpha-1)\eta_{3,\alpha,\delta}}.
\end{align}

In particular, when $\delta=0$, the order of $M$ can be improved and the $c_{2,M}$ can be reduced to $c_{3,M}$, which is defined as follows:
\begin{align*}
c_{3,M}=&
\frac{4d_{\alpha}\mathbb{E}\left[\left|\xi_{1}-\mathbb{E}[\xi_{1}]\right|^{2-\alpha}\right]}{(2-\alpha)(\alpha-1)\sigma^{2-\alpha}}\frac{\eta_{4,\alpha}}{M^{\frac{\alpha^{2}-1}{\alpha^{2}+2\alpha-1}}}
+\frac{3\mathbb{E}\left[|\xi_{1}|\right]\left|\mathbb{E}[\xi_{1}]\right|}{\sigma^{2}}\frac{\eta_{4,\alpha}}{M^{\frac{\alpha^{2}-1}{\alpha^{2}+2\alpha-1}}}\\
&+\eta_{4,\alpha}
\begin{cases}
\frac{2(2A)^{\frac{2}{\alpha}}}{\sigma^{2}}\left[\frac{2}{2-\alpha}+\frac{2L}{\alpha+\gamma-2}(2A)^{-\frac{\alpha+\gamma}{\alpha}}\right]M^{-\frac{\alpha^{2}-1}{\alpha^{2}+2\alpha-1}},&\gamma\in(2-\alpha,\infty),\\
\frac{1}{\sigma^{2}}\left[\left(\frac{4(2A)^{\frac{2}{\alpha}}}{2-\alpha}+\frac{8L}{\alpha-1}\right)
+q_{2}\right]\frac{(\alpha^{2}-1)\log M}{\alpha^{2}+2\alpha-1}M^{-\frac{\alpha^{2}-1}{\alpha^{2}+2\alpha-1}},& \gamma=2-\alpha,\\
\sigma^{\frac{\alpha-\gamma}{\gamma-1}}\left[\left(\frac{4(2A)^{\frac{2}{\alpha}}}{2-\alpha}+\frac{8L}{2-\alpha-\gamma}\right)
+q_{2}\right]M^{-\frac{(\alpha^{2}-1)(\alpha-1)}{(\alpha^{2}+2\alpha-1)(1-\gamma)}}, &\gamma\in(0,2-\alpha),\\
\max\left\{\frac{2\alpha(2A)^{\frac{2}{\alpha}}}{(2-\alpha)\sigma^{2}},\frac{4}{\sigma^{2}},
\frac{8}{(2-\alpha)\sigma^{\alpha}}+\frac{2(2A)^{\frac{2}{\alpha}}}{\sigma^{\alpha}}+q_{2}\right\}M^{-\frac{(\alpha^{2}-1)(\alpha-1)}{\alpha^{2}+2\alpha-1}},
&\gamma=0,
\end{cases}
\end{align*}
with $q_{2}=\frac{8\alpha^{2}(A+L)-4\alpha L}{(\alpha-1)\eta_{4,\alpha}}$.
\end{theorem}

\begin{remark}

(i) In Theorem \ref{thm2}, when $\alpha\in(1,2)$, we have
\begin{align*}
\frac{\alpha^{2}-1}{\alpha^{2}+2\alpha-1}-\frac{2(\alpha-1)}{3\alpha-1}=\frac{(\alpha-1)^{3}}{(\alpha^{2}+2\alpha-1)(3\alpha-1)}>0.
\end{align*}
Hence, the constant $c_{3,M}$ is better than the constant $c_{2,M}$ in terms of the order of $M$.

(ii) The convergence rate of the sample size $n$ in Theorems \ref{thm1} and \ref{thm2} matches the optimal one in the Kolmogorov distance calculated by \cite{Hal81}. In Example \ref{Pareto} below, the simulation result further confirms that the convergence rate is optimal.
\end{remark}

\begin{example}[Pareto distribution]\label{Pareto} Let the random variable $X_{1}$ has the following density function
\begin{align*}
p_{X_{1}}(x)=\frac{\alpha}{2|x|^{\alpha+1}}{\bf 1}_{\{|x|\geq1\}}.
\end{align*}
In this case, $A=\frac{1}{2}$, $\delta=0$, $B(x)=\left(\frac{|x|^{\alpha}}{2}-\frac{1}{2}\right){\bf 1}_{\{|x|\leq1\}}$. We derive from Theorems \ref{thm1} and \ref{thm2} with $\gamma=2$ and $L=\frac{1}{2}$ that $\left|\mathbb{E}\left(S_n-M\right)_{+}-\mathbb{E}\left(S_\alpha (1,\delta)-M\right)_{+}\right|=O(n^{\frac{\alpha-2}{\alpha}}M^{\frac{1-\alpha^{2}}{\alpha^{2}+2\alpha-1}})$.

When $\alpha=1.5$, we check the relation between the sample size $n$ and the difference between the numerical and true probability distributions in the distance of Kolmogorov-Smirnov(KS) test, see Figure \ref{figure:1}(a). Here we use the method from \cite[Theorem 1.3]{N20} to obtain the $\alpha$-stable distribution (with 500 paths), that is, let $\Theta$ and $W$ be independent with $\Theta$ uniformly distributed on $\left(-\frac{\pi}{2},\frac{\pi}{2}\right)$, $W$ exponentially distributed with mean 1, then
\begin{align*}
S_{\alpha}(1,0)=\frac{\sin(\alpha\Theta)}{(\cos\Theta)^{\frac{1}{\alpha}}}\left[\frac{\cos\left((\alpha-1)\Theta\right)}{W}\right]^{\frac{1-\alpha}{\alpha}}.
\end{align*}
We also construct numerical $S_{n}$ with the following pairs of the sample size and the number of paths $(n,N)$ that $(10^{2},10^{2})$, $(10^{3},10^{3})$, $(10^{4},10^{4})$ and $(10^{5},10^{5})$. Moreover, for a fixed 8000 paths, the empirical density functions are presented in Figure \ref{figure:1}(b) for different sample sizes $n=100,500,1000$.

\begin{figure}
  \centering
  \subfigure[The convergence rate along the sample size $n$]{\includegraphics[width=2.9in]{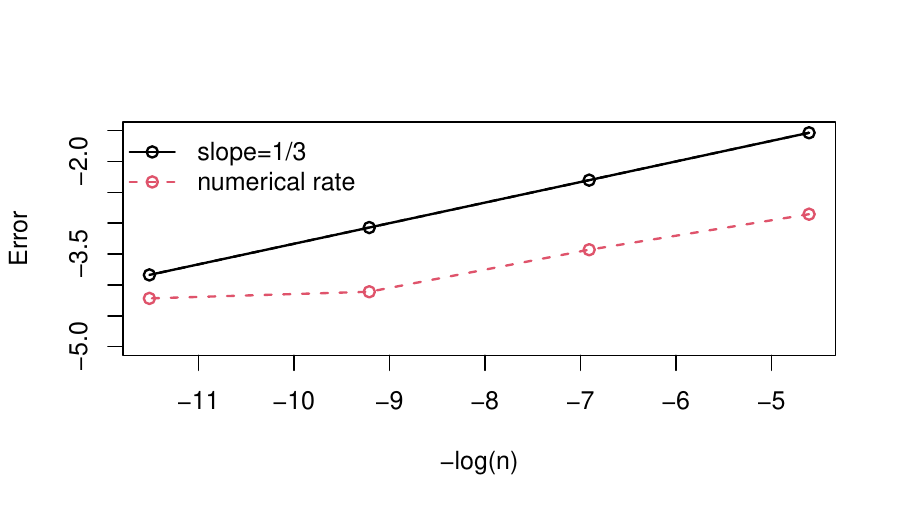}}
  \subfigure[The empirical density functions]{\includegraphics[width=2.9in]{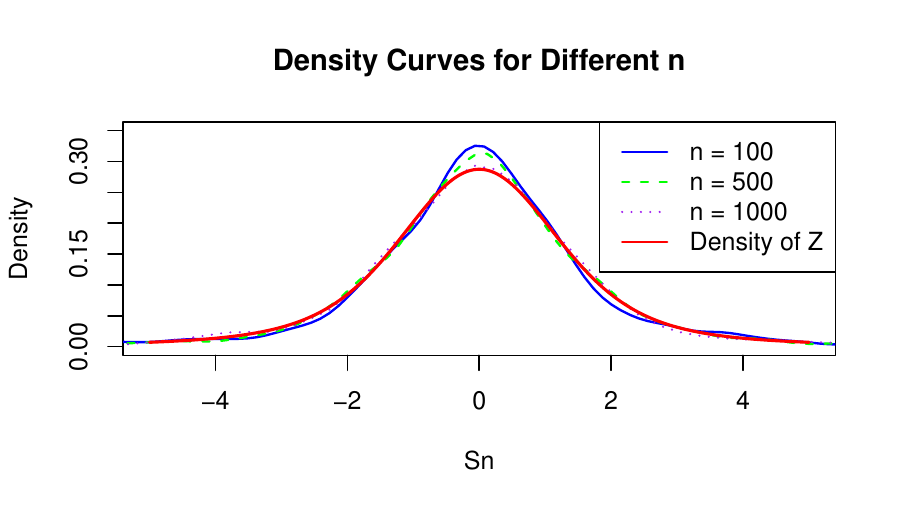}}
  \caption{The Pareto distribution with $\alpha=1.5$}
  \label{figure:1}
 \end{figure}

\end{example}

The paper is organized as follows: In section \ref{Stein}, we establish the Stein's equation and give the regularities of the solution to the equation, containing the uniform and non uniform regularities, which will be crucial to proving our main results. In Section \ref{mainproof}, we use the Stein's method and the Taylor-like extension to prove the Theorems \ref{thm1} and \ref{thm2} above. In Section \ref{conclusion}, some conclusions are summarised, and the future work is discussed. In addition, the proofs of heat kernel estimates used in Section \ref{Stein} and Taylor-like extension appearing in Section \ref{mainproof} are postponed to Appendixes \ref{HKE} and \ref{TLE}, respectively.

\section{Stein's method}\label{Stein}

In this section, we will introduce the Stein's equation, which plays a crucial role in the proof of the stable approximation for call function.

According to the Stein's method developed in \cite{CNX21}, we first give the definition of the infinitesimal generator of the $\alpha$-stable process.
\begin{definition}
For $f:\mathbb{R}\to\mathbb{R}$ in $\mathcal{C}_{b}^2$, define the operator as follows:
\[
(\mathcal{A}^{\alpha,\delta}f)(y)=d_\alpha \,
\int_{\infty}^{\infty} \frac{f(y+u)-f(y)-uf'(y)}{2|u|^{1+\alpha}}
\left[
(1+\delta){\bf 1}_{(0,\infty)}(u)+(1-\delta){\bf 1}_{(-\infty,0)}(u)
\right]
 du,
\]
where $d_\alpha = \left(
\int_0^\infty \frac{1-\cos u}{u^{1+\alpha}}du
\right)^{-1}$.
\end{definition}
Then, the Stein's equation is given as follows:
\begin{equation}\label{stablestein}
\mathcal{L}^{\alpha,\delta} f_g(y):=\mathcal{A}^{\alpha,\delta}f_{g}(y)-\frac{1}{\alpha}yf_{g}'(y)= g(y) - \nu(g),\qquad \forall y\in\mathbb{R},
\end{equation}
where $\nu$ is the distribution of $S_\alpha(1,\delta)$ and $\nu(g)=\int_{-\infty}^{\infty}g(x)\nu(dx)$.

It is well known that the distribution $S_{\alpha}(\sigma,\delta)$ admits a smooth density \cite[Proposition 2.5, (xii)]{sato}, denote it by  $p_{\sigma,\delta}$. Then, the density satisfies the following scaling property
\begin{align}\label{scale}
p_{\sigma,\delta}(y)=\sigma^{-1}p_{1,\delta}\left(\frac{y}{\sigma}\right),
\end{align}
and by \cite[Lemma 2.3]{CNX21}, the solution to equation \eqref{stablestein} is
\begin{equation}\label{phih}
f_g(y) = -\int_0^\infty \int_{-\infty}^{\infty} p_{(1-e^{-t})^{\frac{1}{\alpha}},\delta}(u-e^{-\frac{t}{\alpha}}y)(g(u)-\nu(g))dudt.
\end{equation}

Since the density $p_{\sigma,\delta}$ does not have an explicit expression, we need the following heat kernel estimates, which will be proved in Appendix \ref{HKE}.

\begin{lemma}\label{heat}
Let the random variable $Y\sim S_\alpha(1,\delta)$ and denote the density function of $Y$ by $p_{1,\delta}$. Then for any $y\in(-\infty,\infty)$, we have
\begin{align}\label{den}
p_{1,\delta}(y)\leq\eta_{1,\alpha,\delta}\min\left\{1,\frac{1}{y^{2}}\right\},\qquad
\left|p_{1,\delta}'(y)\right|\leq\frac{\eta_{2,\alpha,\delta}}{{\rm Beta}\left(\frac{2}{\alpha},1-\frac{1}{\alpha}\right)}\min\left\{1,\frac{1}{y^{2}}\right\}.
\end{align}
where $\eta_{1,\alpha,\delta}$ and $\eta_{2,\alpha,\delta}$ are defined by \eqref{eta1} and \eqref{eta2}, respectively.
\end{lemma}

In particular, when $\delta=0$, \cite{CNXY24} give the following heat kernel estimates.

\begin{lemma}\label{heat0}
Let the random variable $Y\sim S_\alpha(1,0)$ and denote the density function of $Y$ by $p_{1,0}$. Then for any $y\in(-\infty,\infty)$, we have
\begin{align}\label{den0}
p_{1,0}(y)\leq\max\left\{\frac{\Gamma(\frac{1}{\alpha})}{\alpha},\frac{\alpha 2^{\alpha-1}\sin\frac{\alpha\pi}{2}\Gamma(\frac{1+\alpha}{2})\Gamma(\frac{\alpha}{2})}{\pi^{\frac{3}{2}}}\right\}\min\left\{1,\frac{1}{|y|^{\alpha+1}}\right\}.
\end{align}
\end{lemma}

Now, we will give the uniform and non uniform regularities of the solution $f_{g}$.

\subsection{Regularities}

\begin{lemma}[Uniform]
For any $g\in{\rm Lip}(1)$ and let $f_{g}$ be the solution to equation \eqref{stablestein}, then for any $y\in(-\infty,\infty)$, we have
\begin{align}\label{unregu}
\|f'_{g}\|_{\infty}\leq\alpha, \qquad \|f''_{g}\|_{\infty}\leq4\eta_{2,\alpha,\delta},
\end{align}
where $\|f\|_{\infty}:=\sup_{x\in\mathbb{R}}|f(x)|$ for any function $f:\mathbb{R}\rightarrow\mathbb{R}$.
\end{lemma}

\begin{proof}
By \cite[Theorem 1.2]{CNX21}, we immediately obtain
\begin{align*}
\|f'_{g}\|_{\infty}\leq\alpha, \qquad \|f''_{g}\|_{\infty}\leq {\rm Beta}\left(\frac{2}{\alpha},1-\frac{1}{\alpha}\right)\int_{-\infty}^{\infty}\left|p'_{1,\delta}(y)\right|d y.
\end{align*}
Hence, we need to derive the upper bound of the integral $\int_{-\infty}^{\infty}\left|p'_{1,\delta}(y)\right|d y$. By \eqref{den}, we have
\begin{align*}
\|f''_{g}\|_{\infty}\leq&\eta_{2,\alpha,\delta}\int_{-\infty}^{\infty}\min\left\{1,\frac{1}{y^{2}}\right\}d y
=2\eta_{2,\alpha,\delta}\left[\int_{0}^{1}1\dif y+\int_{1}^{\infty}\frac{1}{y^{2}}\dif y\right]=4\eta_{2,\alpha,\delta}.
\end{align*}
The proof is complete.
\end{proof}

\begin{lemma}[Non uniform]\label{nonregu}
For any $M>2$ and let $f_{g_{M}}$ be the solution to equation \eqref{stablestein}, then for any $y\in\left(-\infty,\infty\right)$, we have
\begin{align}\label{nonregu2}
\|f''_{g_{M}}(y)\|_{\infty}
\leq\frac{\left(4\alpha+2\alpha^{\frac{2\alpha-2}{\alpha}}-1\right)\eta_{1,\alpha,\delta}}{(\alpha-1)M^{\frac{2(1-\alpha)}{3\alpha-1}}}
:=\frac{\eta_{3,\alpha,\delta}}{M^{\frac{2(\alpha-1)}{3\alpha-1}}}.
\end{align}
\end{lemma}

\begin{proof}
Recall $f_{g_{M}}$ from \eqref{phih}, by a change of variable $v=u-e^{-\frac{t}{\alpha}}y$, we have
\begin{align*}
f_{g_{M}}(y)=&-\int_0^\infty \int_{-\infty}^{\infty} p_{(1-e^{-t})^{\frac{1}{\alpha}},\delta}(v)\left(g_{M}\left(v+e^{-\frac{t}{\alpha}}y\right)-\nu(g)\right)dvdt.
\end{align*}
By the definition of the function $g_{M}$ and the scaling property \eqref{scale}, we have
\begin{align}\label{gradform}
f'_{g_{M}}(y)=&-\int_0^\infty \int_{M-e^{-\frac{t}{\alpha}}y}^{\infty}e^{-\frac{t}{\alpha}} p_{(1-e^{-t})^{\frac{1}{\alpha}},\delta}(v)dvdt\nonumber\\
=&-\int_0^\infty \int_{M-e^{-\frac{t}{\alpha}}y}^{\infty}e^{-\frac{t}{\alpha}}(1-e^{-t})^{-\frac{1}{\alpha}} p_{1,\delta}\left(\frac{v}{(1-e^{-t})^{\frac{1}{\alpha}}}\right)dvdt\nonumber\\
=&-\int_{0}^{1} \int_{M-r^{\frac{1}{\alpha}}y}^{\infty}r^{\frac{1-\alpha}{\alpha}}(1-r)^{-\frac{1}{\alpha}} p_{1,\delta}\left(\frac{v}{(1-r)^{\frac{1}{\alpha}}}\right)dvdr,
\end{align}
where the last equality is by taking $r=e^{-t}$. Then one can write by \eqref{gradform} and \eqref{den}
\begin{align*}
\left|f''_{g_{M}}(y)\right|\leq&\int_{0}^{1}r^{\frac{2-\alpha}{\alpha}}(1-r)^{-\frac{1}{\alpha}} p_{1,\delta}\left(\frac{M-r^{\frac{1}{\alpha}}y}{(1-r)^{\frac{1}{\alpha}}}\right)dr\\
\leq&\eta_{1,\alpha,\delta}\int_{0}^{1}r^{\frac{2-\alpha}{\alpha}}(1-r)^{-\frac{1}{\alpha}}
\left\{1,\frac{(1-r)^{\frac{2}{\alpha}}}{\left(M-r^{\frac{1}{\alpha}}y\right)^{2}}\right\}dr.
\end{align*}
When $y<\frac{M}{2}$, we have $M-r^{\frac{1}{\alpha}}y>\frac{M}{2}$ for any $r\in(0,1)$, this implies
\begin{align}\label{sinte1}
\left|f''_{g_{M}}(y)\right|\leq&\eta_{1,\alpha,\delta}\int_{0}^{1}r^{\frac{2-\alpha}{\alpha}}\frac{(1-r)^{\frac{1}{\alpha}}}{\left(M-r^{\frac{1}{\alpha}}y\right)^{2}}dr
\leq\frac{4\eta_{1,\alpha,\delta}}{M^{2}}.
\end{align}
When $\frac{M}{2}\leq y\leq M$, for some $\theta>0$,
\begin{align*}
&\int_{0}^{1}r^{\frac{2-\alpha}{\alpha}}(1-r)^{-\frac{1}{\alpha}}
\left\{1,\frac{(1-r)^{\frac{2}{\alpha}}}{\left(M-r^{\frac{1}{\alpha}}y\right)^{2}}\right\}dr\\
\leq&\int_{0}^{(1-M^{-\theta})^{\alpha}}r^{\frac{2-\alpha}{\alpha}}\frac{(1-r)^{\frac{1}{\alpha}}}{\left(M-r^{\frac{1}{\alpha}}y\right)^{2}}dr
+\int_{(1-M^{-\theta})^{\alpha}}^{1}(1-r)^{-\frac{1}{\alpha}}dr\\
\leq&\frac{M^{2\theta}}{y^{2}}+\frac{\alpha}{\alpha-1}\left[1-(1-M^{-\theta})^{\alpha}\right]^{\frac{\alpha-1}{\alpha}}
\leq\frac{4M^{2\theta}}{M^{2}}+\frac{\alpha^{\frac{2\alpha-1}{\alpha-1}}}{(\alpha-1)M^{\frac{\theta(\alpha-1)}{\alpha}}},
\end{align*}
by taking $\theta=\frac{2\alpha}{3\alpha-1}$, we have
\begin{align}\label{sinte2}
\int_{0}^{1}r^{\frac{2-\alpha}{\alpha}}(1-r)^{-\frac{1}{\alpha}}
\left\{1,\frac{(1-r)^{\frac{2}{\alpha}}}{\left(M-r^{\frac{1}{\alpha}}y\right)^{2}}\right\}dr\leq\frac{4(\alpha-1)+\alpha^{\frac{2\alpha-1}{\alpha}}}{(\alpha-1)M^{\frac{2(\alpha-1)}{3\alpha-1}}}.
\end{align}
When $y>M$, we divide the integral into following two parts:
\begin{align*}
&\int_{0}^{1}r^{\frac{2-\alpha}{\alpha}}(1-r)^{-\frac{1}{\alpha}}
\left\{1,\frac{(1-r)^{\frac{2}{\alpha}}}{\left(M-r^{\frac{1}{\alpha}}y\right)^{2}}\right\}dr\\
\leq&\left(\int_{0}^{(\frac{M}{y})^{\alpha}}+\int_{(\frac{M}{y})^{\alpha}}^{1}\right)r^{\frac{2-\alpha}{\alpha}}(1-r)^{-\frac{1}{\alpha}}
\left\{1,\frac{(1-r)^{\frac{2}{\alpha}}}{\left(M-r^{\frac{1}{\alpha}}y\right)^{2}}\right\}dr:=\mathcal{J}_{1}+\mathcal{J}_{2}.
\end{align*}
For $\mathcal{J}_{1}$, we have
\begin{align*}
\mathcal{J}_{1}\leq\int_{0}^{((\frac{M}{y}-y^{-\theta})\vee0)^{\alpha}}r^{\frac{2-\alpha}{\alpha}}
\frac{(1-r)^{\frac{1}{\alpha}}}{\left(M-r^{\frac{1}{\alpha}}y\right)^{2}}dr
+\int_{((\frac{M}{y}-y^{-\theta})\vee0)^{\alpha}}^{(\frac{M}{y})^{\alpha}}(1-r)^{-\frac{1}{\alpha}}dr,
\end{align*}
where $\theta>0$ is a constant, which will be chosen later. On the one hand,
\begin{align}\label{sinte3}
\int_{0}^{((\frac{M}{y}-y^{-\theta})\vee0)^{\alpha}}r^{\frac{2-\alpha}{\alpha}}
\frac{(1-r)^{\frac{1}{\alpha}}}{\left(M-r^{\frac{1}{\alpha}}y\right)^{2}}dr\leq y^{2\theta-2}.
\end{align}
On the other hand, we have
\begin{align*}
\int_{((\frac{M}{y}-y^{-\theta})\vee0)^{\alpha}}^{(\frac{M}{y})^{\alpha}}(1-r)^{-\frac{1}{\alpha}}dr
\leq&\frac{\alpha}{\alpha-1}\left[\left(1-((\frac{M}{y}-y^{-\theta})\vee0)^{\alpha}\right)^{\frac{\alpha-1}{\alpha}}
-\left(1-(\frac{M}{y})^{\alpha}\right)^{\frac{\alpha-1}{\alpha}}\right],
\end{align*}
then for some $b>0$, if $1-(\frac{M}{y})^{\alpha}>b$, one can derive from the Taylor expansion that
\begin{align*}
\int_{((\frac{M}{y}-y^{-\theta})\vee0)^{\alpha}}^{(\frac{M}{y})^{\alpha}}(1-r)^{-\frac{1}{\alpha}}dr
\leq&\alpha b^{-\frac{1}{\alpha}}y^{-\theta},
\end{align*}
and if $1-(\frac{M}{y})^{\alpha}\leq b$, we have
\begin{align*}
\int_{((\frac{M}{y}-y^{-\theta})\vee0)^{\alpha}}^{(\frac{M}{y})^{\alpha}}(1-r)^{-\frac{1}{\alpha}}dr
\leq&\frac{\alpha}{\alpha-1}\left[b^{\frac{\alpha-1}{\alpha}}+\alpha^{\frac{\alpha-1}{\alpha}} y^{\frac{-\theta(\alpha-1)}{\alpha}}\right],
\end{align*}
then taking $b=y^{-\theta}$, we have
\begin{align}\label{sinte4}
\int_{((\frac{M}{y}-y^{-\theta})\vee0)^{\alpha}}^{(\frac{M}{y})^{\alpha}}(1-r)^{-\frac{1}{\alpha}}dr
\leq\frac{\alpha+\alpha^{\frac{2\alpha-1}{\alpha}}}{\alpha-1}y^{-\frac{\theta(\alpha-1)}{\alpha}}
\leq\frac{\alpha+\alpha^{\frac{2\alpha-1}{\alpha}}}{(\alpha-1)M^{\frac{2(\alpha-1)}{3\alpha-1}}},
\end{align}
with $\theta=\frac{2\alpha}{3\alpha-1}$. For $\mathcal{J}_{2}$, by the same argument as above, we have
\begin{align}\label{sinte5}
\mathcal{J}_{2}\leq&\int_{(\frac{M}{y})^{\alpha}}^{(\frac{M}{y}+y^{-\theta})^{\alpha}\wedge1}(1-r)^{-\frac{1}{\alpha}}dr
+\int_{(\frac{M}{y}+y^{-\theta})^{\alpha}\wedge1}^{1}r^{\frac{2-\alpha}{\alpha}}
\frac{(1-r)^{\frac{1}{\alpha}}}{\left(M-r^{\frac{1}{\alpha}}y\right)^{2}}dr\nonumber\\
\leq&\frac{\alpha+\alpha^{\frac{2\alpha-1}{\alpha}}}{\alpha-1}y^{-\frac{\theta(\alpha-1)}{\alpha}}+y^{2\theta-2}
\leq\frac{2\alpha+\alpha^{\frac{2\alpha-1}{\alpha}}-1}{(\alpha-1)M^{\frac{2(\alpha-1)}{3\alpha-1}}}.
\end{align}
Combining \eqref{sinte1}, \eqref{sinte2}, \eqref{sinte3}, \eqref{sinte4} and \eqref{sinte5}, we have
\begin{align*}
\left|f''_{g_{M}}(y)\right|
\leq\frac{\left(4\alpha+2\alpha^{\frac{2\alpha-2}{\alpha}}-1\right)\eta_{1,\alpha,\delta}}{(\alpha-1)M^{\frac{2(\alpha-1)}{3\alpha-1}}},
\end{align*}
the proof is complete.
\end{proof}

In particular, when $\delta=0$, with the help of Lemma \ref{heat0}, by the same argument as the proof of Lemma \ref{nonregu}, we have the following lemma.

\begin{lemma}\label{nonregu0}
For any $M>2$ and let $f_{g_{M}}$ be the solution to equation \eqref{stablestein}, then for any $y\in\left(-\infty,\infty\right)$, we have
\begin{align*}
\|f''_{g_{M}}(y)\|_{\infty}
\leq\frac{\eta_{4,\alpha}}{M^{\frac{\alpha^{2}-1}{\alpha^{2}+2\alpha-1}}},
\end{align*}
where $\eta_{4,\alpha}$ is defined by \eqref{6}.
\end{lemma}

\section{Proof of Theorems \ref{thm1} and \ref{thm2}}\label{mainproof}

In this section, under the heavy-tailed setting stated as above, we use the Stein's method to obtain the uniform and nonuniform convergence bounds of the stable approximation for call function.

\subsection{Proof of Theorem \ref{thm1}}

Notice that $g_{M}(x)=(x-M)_{+}$ is Lipschitz continuous with constant $1$, according to the definition of Wasserstein-1 distance and the proof of \cite[Theorem 1.4]{CNX21}, it is easy to verify that
\begin{align*}
\left|\mathbb{E}\left(S_n-M\right)_{+}-\mathbb{E}\left(S_\alpha (1,\beta)-M\right)_{+}\right|\leq c'R_{n},
\end{align*}
where $R_{n}$ is defined by \ref{Rn} and
\begin{align*}
c'=&
\frac{4d_{\alpha}\mathbb{E}\left[\left|\xi_{1}-\mathbb{E}[\xi_{1}]\right|^{2-\alpha}\right]}{(2-\alpha)(\alpha-1)\sigma^{2-\alpha}}\|f_{g}''\|_{\infty}
+\frac{3\mathbb{E}\left[|\xi_{1}|\right]\left|\mathbb{E}[\xi_{1}]\right|}{\sigma^{2}}\|f_{g}''\|_{\infty}\\
&+
\begin{cases}
\frac{2(2A)^{\frac{2}{\alpha}}}{\sigma^{2}}\left[\frac{2}{2-\alpha}+\frac{2L}{\alpha+\gamma-2}(2A)^{-\frac{\alpha+\gamma}{\alpha}}\right]\|f_{g}''\|_{\infty},
&\gamma\in(2-\alpha,\infty),\\
\frac{1}{\sigma^{2}}\left[\left(\frac{4(2A)^{\frac{2}{\alpha}}}{2-\alpha}+\frac{8L}{\alpha-1}\right)\|f_{g}''\|_{\infty}
+\frac{8\alpha(A+L)-4L}{\alpha-1}\|f_{g}'\|_{\infty}\right],& \gamma=2-\alpha,\\
\sigma^{\frac{\alpha-\gamma}{\gamma-1}}\left[\left(\frac{4(2A)^{\frac{2}{\alpha}}}{2-\alpha}+\frac{8L}{2-\alpha-\gamma}\right)\|f''_{g}\|_{\infty}
+\frac{8\alpha(A+L)-4L}{\alpha-1}\|f_{g}'\|_{\infty}\right], &\gamma\in(0,2-\alpha),\\
\max\left\{\frac{2\alpha(2A)^{\frac{2}{\alpha}}}{(2-\alpha)\sigma^{2}},\frac{4}{\sigma^{2}},
\frac{8}{(2-\alpha)\sigma^{\alpha}}+\frac{2(2A)^{\frac{2}{\alpha}}}{\sigma^{\alpha}}+\frac{8\alpha(A+L)-4L}{(\alpha-1)\sigma^{\alpha}}\frac{\|f_{g}'\|_{\infty}}{\|f_{g}''\|_{\infty}}\right\}\|f_{g}''\|_{\infty},\!\!\!
&\gamma=0.
\end{cases}
\end{align*}
Combining \eqref{unregu}, Theorem \ref{thm1} is proved.
\qed

\subsection{Proof of Theorem \ref{thm2}}

Before giving the proof of Theorem \ref{thm2}, we need the following Taylor-like extension, which will be proved in Appendix \ref{TLE}.

\begin{lemma}\label{Taylor}
Let $X$ have a distribution of the form (\ref{px}) and $B(x)$ satisfy (\ref{bound}) with $\gamma\in(0,2-\alpha]$. Let
$Y$ be an integrable random variable, which is independent of $X$. For any $0<a<(2A)^{-\frac{1}{\alpha}}\wedge1$, denote
\begin{align*}
T=&\left|
\mathbb{E}[Xf_{g_{M}}'(Y+aX)] -\mathbb{E}[X]\mathbb{E}[f_{g_{M}}'(Y)]-\frac{2A\alpha^{2}}{d_\alpha}\,a^{\alpha-1}\mathbb{E}\left[\big(\mathcal{A}^{\alpha,\delta}f_{g_{M}}\big)(Y)\right]
\right|.
\end{align*}
Then:
\begin{enumerate}
\item[i)] When $\gamma=2-\alpha,$ we have
\begin{align*}
T\leq\frac{2\alpha(2A)^{\frac{2}{\alpha}}\eta_{3,\alpha,\delta}}{(2-\alpha)M^{\frac{\alpha^{2}-1}{\alpha^{2}+2\alpha-1}}}a
+&\eta_{3,\alpha,\delta}\Big[\Big(2(2A)^{\frac{2}{\alpha}}+\frac{8L}{\alpha-1}\Big)+q_{1}\Big]\frac{2(\alpha-1)a\,|\log a|\log M}{(3\alpha-1)M^{\frac{2(\alpha-1)}{3\alpha-1}}}.
\end{align*}
\item[ii)] When $\gamma\in(0,2-\alpha),$ we have
\begin{align*}
T\leq\eta_{3,\alpha,\delta}\Big[\Big(\frac{4(2A)^{\frac{2}{\alpha}}}{2-\alpha}+\frac{8L}{2-\alpha-\gamma}\Big)+q_{1}\Big]a^{\frac{1-\alpha}{\gamma-1}}M^{-\frac{2(\alpha-1)^{2}}{(3\alpha-1)(1-\gamma)}}.
\end{align*}
\item[iii)] When $\gamma=0,$ we have
\begin{align*}
T\leq&\frac{2\alpha(2A)^{\frac{2}{\alpha}}\eta_{4,\alpha}}{(2-\alpha)M^{\frac{\alpha^{2}-1}{\alpha^{2}+2\alpha-1}}}a
+4\eta_{3,\alpha,\delta}aM^{-\frac{2(\alpha-1)^{2}}{(3\alpha-1)(1-\gamma)}}\int_{-a^{-1}}^{a^{-1}}\frac{|B(x)|}{|x|^{\alpha-1}}dx\\
&+\eta_{3,\alpha,\delta}\Big[\Big(\frac{8}{2-\alpha}+2(2A)^{\frac{2}{\alpha}}\Big)+q_{1}\Big]a^{\alpha-1}\big(\sup_{|x|\geq a^{-1}}|B(x)|\big)^{\alpha-1}M^{-\frac{2(\alpha-1)^{2}}{(3\alpha-1)}},
\end{align*}
\end{enumerate}
where $q_{1}$ is defined by \eqref{q1} above.
\end{lemma}

Now, we are in the position to give the proof of Theorem \ref{thm2}.

\begin{proof}[Proof of Theorem \ref{thm2}]
According to the definition of the constant $c'$, by the same argument as the proof of \cite[Theorem 1.4]{CNX21}, Theorem \ref{thm2} follows from Lemma \ref{nonregu} and Lemma \ref{Taylor} immediately. In particular, when $\delta=0$, similar to the calculation of the constant $c_{2,M}$ with Lemma \ref{nonregu} replaced by Lemma \ref{nonregu0}, the constant $c_{3,M}$ can be derived directly.
\end{proof}

\section{Conclusions}\label{conclusion}
In this paper, we study the $\alpha$-stable approximation for the call function and obtain two types of convergence rates. In the progress of obtaining the non uniform bound, we adopt the truncation method, which lead to a convergence rate for $M$ that may not be optimal. Hence, the better convergence rate is worth further investigation. In addition, we will evaluate the practical performance of our results for the data from CDO pricing models and other financial models.

\section*{Acknowledgements}
Peng Chen's research was supported by the National Natural Science Foundation of China (12301176)
and Natural Science Foundation of Jiangsu Province ( BK20220867).
Ting Zhang's research was supported by the National Natural
Science Foundation of China (12301342).

\begin{appendix}

\section{Heat kernel estimates}\label{HKE}

In this section, we will abbreviate the density $p_{1,\delta}$ as $p$ and give the proof of Lemma \ref{heat}.

\begin{proof}[Proof of Lemma \ref{heat}]
Recall
\begin{equation*}
	E[e^{i\lambda Y}]=
	exp\left\{-|\lambda|^\alpha\left(1-i\,\delta\, {\rm sign}(\lambda)\,{\rm tan}\frac{\pi\alpha}{2}\right)\right\}
\end{equation*}
and by the inverse of Fourier transform, we have
\begin{align*}
p(y)=&\frac{1}{2\pi}\int_{-\infty}^{+\infty}e^{-|\lambda|^\alpha(1-i\,\delta\, {\rm sign}(\lambda)\,tan\frac{\pi\alpha}{2})}e^{-i\lambda y}d\lambda\\
=&\frac{1}{2\pi}\int_{-\infty}^{+\infty}e^{-|\lambda|^\alpha}e^{i(|\lambda|^\alpha\delta {\rm sign}(\lambda) {\rm tan}\frac{\pi\alpha}{2}-\lambda y)}d\lambda \\
=&\frac{1}{2\pi}\int_{-\infty}^{+\infty}e^{-|\lambda|^\alpha} \cos(|\lambda|^\alpha\delta {\rm sign}(\lambda) {\rm tan}\frac{\pi\alpha}{2}-\lambda y)d\lambda \\
=&\frac{1}{\pi}\int_{0}^{+\infty}e^{-\lambda^\alpha}\cos(\lambda^\alpha\delta\,{\rm tan}\frac{\pi\alpha}{2}-\lambda y)d\lambda.
\end{align*}
Hence, for $r\in(-1,\infty)$, we denote
\begin{align}\label{I}
\mathcal{I}_{r}(y)=\int_{0}^{\infty}\lambda^{r} e^{-\lambda^\alpha}\cos(\lambda^\alpha\,\delta\, {\rm tan}\frac{\pi\alpha}{2}-\lambda y)d \lambda,
\end{align}
\begin{align}\label{J}
\mathcal{J}_{r}(y)=\int_{0}^{\infty}\lambda^{r} e^{-\lambda^\alpha}\sin(\lambda^\alpha\,\delta\, {\rm tan}\frac{\pi\alpha}{2}-\lambda y)d \lambda,
\end{align}
then
\begin{align}\label{form1}
p(y)=\frac{\mathcal{I}_{0}(y)}{\pi}.
\end{align}
By a change of variable $u=\lambda^\alpha$ and the fact $|\cos(x)|\leq1$, $|\sin(x)|\leq1$ for any $x\in(-\infty,\infty)$, it is straightforward to calculate that
\begin{align}\label{tin}
|\mathcal{I}_{r}(y)|\leq\frac{\Gamma(\frac{r+1}{\alpha})}{\alpha},\qquad |\mathcal{J}_{r}(y)|\leq\frac{\Gamma(\frac{r+1}{\alpha})}{\alpha}.
\end{align}
In addition, notice that
\begin{align*}
\mathcal{I}_{r}(y)=&\int_{0}^{\infty}\lambda^{r}e^{-\lambda^\alpha}\cos(\lambda^\alpha\delta\,{\rm tan}\frac{\pi\alpha}{2}-\lambda y)d\lambda \\
=&\int_{0}^{\infty}\lambda^{r}e^{-\lambda^\alpha}\cos(\lambda^\alpha\delta\,{\rm tan}\frac{\pi\alpha}{2})\cos(\lambda y)d\lambda+\int_{0}^{\infty}\lambda^{r}e^{-\lambda^\alpha}\sin(\lambda^\alpha\delta\,{\rm tan}\frac{\pi\alpha}{2})\sin(\lambda y)d\lambda,
\end{align*}
when $y\neq0$, one can derive from the integration by parts that
\begin{align}\label{I1}
\mathcal{I}_{r}(y)=\frac{r\mathcal{J}_{r-1}(y)}{y}-\frac{\alpha\mathcal{J}_{\alpha+r-1}(y)}{y}+\frac{\alpha\delta\,{\rm tan}\frac{\pi\alpha}{2}\mathcal{I}_{\alpha+r-1}(y)}{y}.
\end{align}
Similarly, we have
\begin{align}\label{J1}
\mathcal{J}_{r}(y)=\frac{-r\mathcal{I}_{r-1}(y)}{y}+\frac{\alpha\mathcal{I}_{\alpha+r-1}(y)}{y}+\frac{\alpha\delta\,{\rm tan}\frac{\pi\alpha}{2}\mathcal{J}_{\alpha+r-1}(y)}{y}.
\end{align}

Now, we first prove the upper bound of $p(y)$. By \eqref{form1} and \eqref{tin}, we have
\begin{align}\label{den1}
p(y)=\frac{\mathcal{I}_{0}(y)}{\pi}\leq\frac{\Gamma(\frac{1}{\alpha})}{\pi\alpha}.
\end{align}
By \eqref{I1} and \eqref{J1}, we further have
\begin{align*}
\mathcal{I}_{0}(y)=&-\frac{\alpha\mathcal{J}_{\alpha-1}(y)}{y}+\frac{\alpha\delta\,{\rm tan}\frac{\pi\alpha}{2}\mathcal{I}_{\alpha-1}(y)}{y}\\
=&\frac{\alpha(\alpha-1)\mathcal{I}_{\alpha-2}(y)}{y^{2}}-\frac{\alpha^{2}\mathcal{I}_{2\alpha-2}(y)}{y^{2}}-\frac{\alpha^{2}\delta\,{\rm tan}\frac{\pi\alpha}{2}\mathcal{J}_{2\alpha-2}(y)}{y^{2}}\\
&+\frac{\alpha\delta\,{\rm tan}\frac{\pi\alpha}{2}\left[(\alpha-1)\mathcal{J}_{\alpha-2}(y)-\alpha\mathcal{J}_{2\alpha-2}(y)+\alpha\delta\,{\rm tan}\frac{\pi\alpha}{2}\mathcal{I}_{2\alpha-2}(y)\right]}{y^{2}},
\end{align*}
which implies from \eqref{tin} that
\begin{align*}
\left|\mathcal{I}_{0}(y)\right|\leq&\frac{(\alpha-1)\left(1+\delta\,{\rm tan}\frac{\pi\alpha}{2}\right)\Gamma\left(\frac{\alpha-1}{\alpha}\right)+\alpha\left(1+\delta\,{\rm tan}\frac{\pi\alpha}{2}\right)^{2}\Gamma\left(\frac{2\alpha-1}{\alpha}\right)}{y^{2}}\\
=&\frac{(\alpha-1)\left(1+\delta\,{\rm tan}\frac{\pi\alpha}{2}\right)\left(2+\delta\,{\rm tan}\frac{\pi\alpha}{2}\right)\Gamma\left(\frac{\alpha-1}{\alpha}\right)}{y^{2}},
\end{align*}
then \eqref{den} follows from \eqref{form1} and \eqref{den1}.

For the second inequality, one can derive from \eqref{form1} and \eqref{tin} that
\begin{align}\label{constant}
\left|p'(y)\right|=\frac{\left|\mathcal{I}_0'(y)\right|}{\pi}=\frac{\left|\mathcal{J}_1(y)\right|}{\pi}
\leq\frac{\Gamma(\frac{2}{\alpha})}{\pi\alpha},
\end{align}
whereas by \eqref{I1} and \eqref{J1}, we have
\begin{align*}
\mathcal{I}_{0}'(y)=&-\frac{\alpha\mathcal{J}_{\alpha-1}'(y)}{y}+\frac{\alpha\mathcal{J}_{\alpha-1}(y)}{y^2}+\frac{\alpha\delta\,{\rm tan}\frac{\pi\alpha}{2}\mathcal{I}_{\alpha-1}'(y)}{y}-\frac{\alpha\delta\,{\rm tan}\frac{\pi\alpha}{2}\mathcal{I}_{\alpha-1}(y)}{y^2}\\
=&\frac{\alpha\mathcal{I}_{\alpha}(y)}{y}+\frac{\alpha\mathcal{J}_{\alpha-1}(y)}{y^2}+\frac{\alpha\delta\,{\rm tan}\frac{\pi\alpha}{2}\mathcal{J}_{\alpha}(y)}{y}-\frac{\alpha\delta\,{\rm tan}\frac{\pi\alpha}{2}\mathcal{I}_{\alpha-1}(y)}{y^2}\\
=&\frac{\alpha^{2}\mathcal{J}_{\alpha-1}(y)}{y^{2}}-\frac{\alpha^{2}\mathcal{J}_{2\alpha-1}(y)}{y^{2}}+\frac{\alpha^{2}\delta\,{\rm tan}\frac{\pi\alpha}{2}\mathcal{I}_{2\alpha-1}(y)}{y^{2}}+\frac{\alpha\mathcal{J}_{\alpha-1}(y)}{y^2}\\
&+\frac{\alpha\delta\,{\rm tan}\frac{\pi\alpha}{2}\left[-\alpha\mathcal{I}_{\alpha-1}(y)+\alpha\mathcal{I}_{2\alpha-1}(y)+\alpha\delta\,{\rm tan}\frac{\pi\alpha}{2}\mathcal{J}_{2\alpha-1}(y)-\mathcal{I}_{\alpha-1}(y)\right]}{y^{2}}.
\end{align*}
Hence, \eqref{tin} implies
\begin{align*}
\left|\mathcal{I}_{0}'(y)\right|\leq&\frac{(1+\alpha)\left(1+\delta\,{\rm tan}\frac{\pi\alpha}{2}\right)\Gamma(1)+\alpha\left(1+\delta\,{\rm tan}\frac{\pi\alpha}{2}\right)^{2}\Gamma(2)}{y^{2}}\\
=&\frac{\left(1+\delta\,{\rm tan}\frac{\pi\alpha}{2}\right)\left(1+2\alpha+\alpha\delta\,{\rm tan}\frac{\pi\alpha}{2}\right)}{y^{2}},
\end{align*}
together with \eqref{form1} and \eqref{constant}, the desired result follows.
\end{proof}

\section{Taylor-like extension}\label{TLE}

In this section, we will use the zero-biased coupling to obtain the Taylor-like extension.

\begin{proof}[Proof of Lemma \ref{Taylor}]
Similar to the proof of \cite[Lemma 3.3]{CNX21}, define a random variable $\tilde{X}$, which is independent of $Y$ and satisfies
\begin{align}\label{F}
\mathbb{P}(\tilde{X}>x)=\frac{A(1+\delta)}{|x|^{\alpha}},\quad x\geq(2A)^{\frac{1}{\alpha}},\qquad \mathbb{P}(\tilde{X}\leq x)=\frac{A(1-\delta)}{|x|^{\alpha}},\quad x\leq-(2A)^{\frac{1}{\alpha}}.
\end{align}
Denoting the distribution functions of $\tilde{X}$ by $F_{\tilde{X}}$, one can derive that
\begin{align*}
F_{\xi}(x)-F_{\tilde{\xi}}(x)=&\big(\frac{1}{2}-\frac{A+B(x)}{|x|^{\alpha}}\big)(1+\delta){\bf 1}_{(0,(2A)^{\frac{1}{\alpha}})}(x)-\frac{B(x)}{|x|^{\alpha}}(1+\delta){\bf 1}_{((2A)^{\frac{1}{\alpha}},\infty)}(x)\\
&+\big(\frac{A+B(x)}{|x|^{\alpha}}-\frac{1}{2}\big)(1-\delta){\bf 1}_{(-(2A)^{\frac{1}{\alpha}},0)}(x)+\frac{B(x)}{|x|^{\alpha}}(1-\delta){\bf 1}_{(-\infty,-(2A)^{\frac{1}{\alpha}})}(x).
\end{align*}
Then, according to the proof of \cite[Lemma 3.3]{CNX21}, we have
\begin{align}\label{result}
&\Big|\mathbb{E}[Xf_{g_{M}}'(Y+aX)] -\mathbb{E}[X]\mathbb{E}[f_{g_{M}}'(Y)]-\frac{2A\alpha^{2}}{d_\alpha}\,a^{\alpha-1}\mathbb{E}\left[\big(\mathcal{A}^{\alpha,\delta}f_{g_{M}}\big)(Y)\right]\Big|\nonumber\\
\leq&\mathbb{E}\Big|\int_{-\infty}^{\infty}x\big[f_{g_{M}}'(Y+ax)-f_{g_{M}}'(Y)\big]
d\big(F_{X}(x)-F_{\tilde{X}}(x)\big)\Big|+|R|,
\end{align}
where
\begin{align*}
R=\mathbb{E}\Big[\int_{-(2A)^{\frac{1}{\alpha}}}^{(2A)^{\frac{1}{\alpha}}}u\big[f_{g_{M}}'(Y+au)-f_{g_{M}}'(Y)\big]\frac{A\alpha\big[(1+\delta){\bf 1}_{(0,\infty)}(u)+(1-\delta){\bf 1}_{(-\infty,0)}(u)\big]}
{|u|^{\alpha+1}}du\Big].
\end{align*}
It is easy to verify from Lemma \ref{nonregu}
\begin{align*}
|R|\leq& A\alpha\|f_{g_{M}}''\|_{\infty}a\int_{-(2A)^{\frac{1}{\alpha}}}^{(2A)^{\frac{1}{\alpha}}}\frac{(1+\delta){\bf 1}_{(0,\infty)}(u)+(1-\delta){\bf 1}_{(-\infty,0)}(u)}
{|u|^{\alpha-1}}du\\
=&\frac{2\alpha}{2-\alpha}(2A)^{\frac{2}{\alpha}}\|f_{g_{M}}''\|_{\infty}a
\leq\frac{2\alpha(2A)^{\frac{2}{\alpha}}\eta_{3,\alpha,\delta}}{(2-\alpha)M^{\frac{\alpha^{2}-1}{\alpha^{2}+2\alpha-1}}}a.
\end{align*}
Now, let us deal with the first term of (\ref{result}). When $0\leq\gamma\leq2-\alpha$. Choose a number $N>a^{-1}$ and $\theta>0$, which will be chosen later. One has by \cite[Lemma 2.8]{lihu} and using that $|B(x)|\leq L$ for $|x|>M^{\theta}N$,
\begin{align*}
&\mathbb{E}\Big|\int_{|x|>M^{\theta}N}x\big[f_{g_{M}}'(Y+ax)-f_{g_{M}}'(Y)\big]
d\big(F_{X}(x)-F_{\tilde{X}}(x)\big)\Big|\\
\leq&2\|f_{g_{M}}'\|_{\infty}\Big[\int_{|x|>M^{\theta}N}|x|dF_{X}(x)+\int_{|x|>M^{\theta}N}|x|dF_{\tilde{X}}(x)\Big]\\
=&2\|f_{g_{M}}'\|_{\infty}\mathbb{E}\big[|X|{\bf 1}_{(M^{\theta}N,\infty)}(|X|)+|\tilde{X}|{\bf 1}_{(M^{\theta}N,\infty)}(|\tilde{X}|)\big]\leq\frac{4(2A+K)\alpha^{2}}{\alpha-1}M^{\theta(1-\alpha)}N^{1-\alpha},
\end{align*}
where the last inequality is by \eqref{unregu}. On the other hand, by integrating by parts, \eqref{unregu} and Lemma \ref{nonregu},
\begin{align*}
&\mathbb{E}\Big|\int_{-M^{\theta}N}^{M^{\theta}N}x\big[f_{g_{M}}'(Y+ax)-f_{g_{M}}'(Y)\big]
d\big(F_{X}(x)-F_{\tilde{X}}(x)\big)\Big|\\
\leq&4L\|f_{g_{M}}'\|_{\infty}M^{\theta(1-\alpha)}N^{1-\alpha}+2\|f_{g_{M}}''\|_{\infty}a\int_{-M^{\theta}N}^{M^{\theta}N}\big|xF_{X}(x)-xF_{\tilde{X}}(x)\big|dx\\
\leq&4\alpha LM^{\theta(1-\alpha)}N^{1-\alpha}+\frac{2\eta_{3,\alpha,\delta}}{M^{\frac{2(\alpha-1)}{3\alpha-1}}}a\big[(2A)^{\frac{2}{\alpha}}+2\int_{(2A)^{\frac{1}{\alpha}}}^{M^{\theta}N}\frac{|B(x)|}{x^{\alpha-1}}dx
+2\int^{-(2A)^{\frac{1}{\alpha}}}_{-M^{\theta}N}\frac{|B(x)|}{|x|^{\alpha-1}}dx\big].
\end{align*}
If $0<\gamma\leq2-\alpha,$ we have
\begin{align*}
\int_{(2A)^{\frac{1}{\alpha}}}^{M^{\theta}N}\frac{|B(x)|}{x^{\alpha-1}}dx\leq\int_{(2A)^{\frac{1}{\alpha}}}^{M^{\theta}N}\frac{L}{x^{\alpha+\gamma-1}}dx\leq
\begin{cases}
L\log(M^{\theta}N),\quad &\gamma=2-\alpha,\\
\frac{L}{2-\alpha-\gamma}M^{\theta(2-\alpha-\gamma)}N^{2-\alpha-\gamma},\quad &\gamma\in(0,2-\alpha).
\end{cases}
\end{align*}
If $\gamma=0,$ we have
\begin{align*}
\int_{(2A)^{\frac{1}{\alpha}}}^{M^{\theta}N}\!\frac{|B(x)|}{x^{\alpha-1}}dx
=&\int_{(2A)^{\frac{1}{\alpha}}}^{a^{-1}}\!\frac{|B(x)|}{x^{\alpha-1}}dx
+\int_{a^{-1}}^{M^{\theta}N}\!\frac{|B(x)|}{x^{\alpha-1}}dx\\
\leq&\int_{0}^{a^{-1}}\!\frac{|B(x)|}{x^{\alpha-1}}dx+\frac{1}{2-\alpha}\sup_{x>a^{-1}}|B(x)|M^{\theta(2-\alpha)}N^{2-\alpha}.
\end{align*}
Since similar bounds hold true for $\int^{-(2A)^{\frac{1}{\alpha}}}_{-N}\frac{|B(x)|}{|x|^{\alpha-1}}dx,$ we can consider
\begin{align*}
N^{1-\alpha}=
\begin{cases}
aN^{2-\alpha-\gamma},\quad &\gamma\in(0,2-\alpha]\\
a\sup_{|x|>a^{-1}}|B(x)|N^{2-\alpha},\quad &\gamma=0,
\end{cases}
\end{align*}
and for any $\gamma\in[0,2-\alpha]$,
\begin{align*}
M^{\theta(1-\alpha)+\frac{2(\alpha-1)}{3\alpha-1}}=M^{\theta(2-\alpha-\gamma)},
\end{align*}
which implies
\begin{align*}
N=
\begin{cases}
a^{\frac{1}{\gamma-1}},\quad &\gamma\in(0,2-\alpha],\\
a^{-1}\big(\sup_{|x|>a^{-1}}|B(x)|)^{-1},\quad &\gamma=0.
\end{cases}
\end{align*}
and
\begin{align*}
\theta=\frac{2(\alpha-1)}{(3\alpha-1)(1-\gamma)}.
\end{align*}
The desired conclusion follows.
\end{proof}

\end{appendix}

\bibliographystyle{amsplain}

\begin{thebibliography}{99}\frenchspacing

\bibitem{ACHL12} Amraoui, S., Cousot, L., Hitier, S. and Laurent, J. P. (2012). Pricing CDOs with state-dependent stochastic recovery rates. {\it Quantitative Finance}. {\bf 12}(8), pp. 1219-1240.

\bibitem{AH19} Arras, B. and Houdr\'e, C. (2019). On Stein's method for infinitely divisible laws with finite first moment. Springer International Publishing.

 \bibitem{AH192} Arras, B. and Houdr\'e, C. (2019). On Stein's method for multivariate self-decomposable laws with finite first moment. {\it Electronic Journal of Probability}. {\bf 24}(29), pp. 1-63.

 \bibitem{AFV10} Asimit, A. V., Furman, E. and Vernic, R. (2010). On a multivariate pareto distribution. {\it Insurance: Mathematics and Economics}. {\bf 46}(2), pp. 308-316.

     \bibitem{banis}
Banis, I. I. (1973). Estimation of the rate of convergence in the metric $L_{p}$ in the case of a limiting stable law. {\it Mathematical transactions of the Academy of Sciences of the Lithuanian SSR}. {\bf 13}(3), pp. 379-384.

\bibitem{BRZ24} Barbour, A. D., Ross, N. and Zheng, G. (2024). Stein's method, smoothing and functional approximation. {\it Electronic Journal of Probability}. {\bf 29}, pp. 1-29.

\bibitem{BU20} Barman, K. and Upadhye, N. S. (2020). Stein's method for tempered stable distributions. arxiv preprint arxiv:2008.05818.

 \bibitem{Bon20} Bonis, T. (2020). Stein's method for normal approximation in Wasserstein distances with application to the multivariate central limit theorem. {\it Probability Theory and Related Fields}. {\bf 178}(3), pp. 827-860.

\bibitem{Cha07} Chatterjee, S. (2007). Stein's method for concentration inequalities. {\it Probability Theory and
 Related Fields}. {\bf 138}, pp. 305-321.

\bibitem{CGS10} Chen, L. H., Goldstein, L. and Shao, Q. M. (2010). Normal approximation by Stein's method. Springer Science \& Business Media.
	
\bibitem{PJYT24}
Chen, P., Liu, J., Lu, Y. and Zhang, T. (2024). Normal approximation for call function by refined Lindeberg principle. {\it Communications in Statistics-Theory and Methods}. 1-18. https://doi.org/10.1080/03610926.2024.2369312.

\bibitem{CNX21}
Chen, P., Nourdin, I. and Xu, L. (2021). Stein's method for asymmetric $\alpha$-stable distributions, with application to the stable CLT. {\it Journal of Theoretical Probability}. {\bf 34}(3), pp. 1382-1407.

\bibitem{CNXY24}
Chen, P., Nourdin, I., Xu, L. and Yang, X. (2024). Multivariate stable approximation by Stein's method. {\it Journal of Theoretical Probability}. {\bf 37}(1), pp. 446-488.

\bibitem{CNXYZ22}
Chen, P., Nourdin, I., Xu, L., Yang, X. and Zhang, R. (2022). Non-integrable stable approximation by Stein's method. {\it Journal of Theoretical Probability}. {\bf 35}, pp. 1137-1186.

\bibitem{CZHW24} Chen, S., Zhao, Y., Huang, F. W., Wang, B. and Lin, J. H. (2024). Carbon leakage perspective: Unveiling policy dilemmas in emission trading and carbon tariffs under insurer green finance. {\it Energy Economics}. {\bf 130}, 107292.

\bibitem{DN02} Davydov, Y. and Nagaev, A. V. (2002). On two aproaches to approximation of multidimensional stable laws. {\it Journal of multivariate analysis}. {\bf 82}(1), pp. 210-239.

\bibitem{NY09}
El Karoui, N. and Jiao, Y. (2009). Stein's method and zero bias transformation for CDO tranche pricing. {\it Finance and Stochastics}. {\bf 13}, pp. 151-180.

\bibitem{NYD08}
El Karoui, N., Jiao, Y. and Kurtz, D. (2008). Gaussian and Poisson approximation: applications to CDO tranche pricing. {\it Journal of Computational Finance}. {\bf 12}(2), pp. 31-59.

\bibitem{FLS24} Fang, X., Liu, S. H. and Shao, Q. M. (2024). Normal approximation for exponential random graphs. arxiv preprint arxiv:2404.01666.

\bibitem{FFP23} Favaro, S., Fortini, S. and Peluchetti, S. (2023). Deep stable neural networks: large-width asymptotics and convergence rates. {\it Bernoulli}. {\bf 29}(3), pp. 2574-2597.

\bibitem{HDC24} Huang, L., Decreusefond, L. and Coutin, L. (2024). Rate of convergence in the functional central limit theorem for stable processes. arxiv preprint arxiv:2401.16834.

\bibitem{Jag84} Jagannathan, R. (1984). Call options and the risk of underlying securities. {\it Journal of Financial Economics}. {\bf 13}(3), pp. 425-434.

\bibitem{JLL23} Jung, P., Lee, H., Lee, J. and Yang, H. (2023). $\alpha$-Stable convergence of heavy-/light-tailed infinitely wide neural networks. {\it Advances in Applied Probability}. {\bf 55}(4), pp. 1415-1441.

\bibitem{AN21}
Kumar, A. N. (2024). Bounds on negative binomial approximation to call function. {\it REVSTAT-Statistical Journal}. {\bf 22}(1), pp. 25-43.

\bibitem{M63} Mandelbrot, B. (1963). New methods in statistical economics. {\it Journal of political economy}. {\bf 71}(5), pp. 421-440.

\bibitem{KN20}
Neammanee, K. and Yonghint, N. (2020). Poisson approximation for call function via Stein-Chen method. {\it Bulletin of the Malaysian Mathematical Sciences Society}. {\bf 43}(2), pp. 1135-1152.

\bibitem{N14} Nolan, J. P. (2014). Financial modeling with heavy-tailed stable distributions. {\it Wiley Interdisciplinary Reviews: Computational Statistics}. {\bf 6}(1), pp. 45-55.

\bibitem{N20}Nolan, J. P. (2020). Univariate stable distributions: models for heavy tailed data. In Springer Series in Operations
Research and Financial Engineering, Springer, Cham.

\bibitem{PPU21} Pertaia, G., Prokhorov, A. and Uryasev, S. (2021). A new approach to credit ratings. {\it Journal of Banking \& Finance}. {\bf 140}, 106097.

\bibitem{NP09} Nourdin, I. and Peccati, G. (2009). Stein's method on Wiener chaos. {\it Probability Theory and Related Fields}. {\bf 145}, pp. 75-118.

\bibitem{Hal81} Hall, P. (1981). Two-sided bounds on the rate of convergence to a stable law. {\it Probability Theory and Related Fields}.
{\bf 57}(3), pp. 349-364.

\bibitem{RM00} Rachev, S. and Mitnik, S. (2000). Stable paretian models in finance. Wiley, New York.

\bibitem{Rol22} Rollin, A. (2022). Kolmogorov bounds for the normal approximation of the number of triangles in the Erdos-Renyi random graph. {\it Probability in the Engineering and Informational Sciences}. {\bf 36}(3), pp. 747-773.

\bibitem{sato}
Sato, K. I. (1999).
L\'evy processes and infinitely divisible distribution.
Cambridge Studies in Advances Mathematics 68,
Cambridge University Press.

\bibitem{Son20} Song, Y. (2020). Normal approximation by Stein's method under sublinear expectations. {\it Stochastic Processes and their Applications}. {\bf 130}(5), pp. 2838-2850.

\bibitem{Ste72} Stein, C. (1972). A bound for the error in the normal approximation to the distribution of a sum of
dependent random variables. In: Proc. Sixth Berkeley Symp. Math. Statist. Probab. II. Probability
Theory, Univ. California Press, Berkeley, Calif., 583-602.

\bibitem{VX24} Vasquez, A. and Xiao, X. (2024). Default risk and option returns. {\it Management Science}. {\bf 70}(4), pp. 2144-2167.

\bibitem{WCWK12} Wang, Q., Chu, B., Wang, J. and Kumakiri, Y. (2012). Risk analysis of supply contract with call options for buyers. {\it International Journal of Production Economics}. {\bf 139}(1), pp. 97-105.

\bibitem{lihu}
Xu, L. (2019). Approximation of stable law in Wasserstein-1 distance by Stein's method. {\it The Annals of Applied Probability}. {\bf 29}(1), pp. 458-504.

\bibitem{NK24}
Yonghint, N. and Neammanee, K. (2024). Poisson approximation for the expectation of call function with application in collateralized debt obligation. {\it Communications in Statistics-Theory and Methods}. {\bf 53}(14), pp. 5265-5279.

\bibitem{NKN22}
Yonghint, N., Neammanee, K. and Chaidee, N. (2022). Poisson approximation for locally dependent CDO. {\it Communications in Statistics-Theory and Methods}. {\bf 51}(7), pp. 2073-2081.

\end{thebibliography}

\end{document}